%% file: article.tex
\documentclass[a4paper,12pt]{amsart}
\usepackage[hmargin=2.5cm,vmargin=2.5cm]{geometry}
\usepackage{hyperref}
\usepackage{graphicx}
\usepackage{amsmath}
\usepackage{amssymb}
\usepackage{bbm}
\usepackage{dsfont}
\usepackage{enumitem}
\hypersetup{colorlinks=true}
\usepackage{color}

\newtheorem{thm}{Theorem}

\newtheorem{lemma}[thm]{Lemma}

\newtheorem{conjecture}{Conjecture}

\renewcommand{\P}{\mathbf P}

\allowdisplaybreaks

\begin{document}

\author{R\'eka Szab\'o, Daniel Valesin}
\address{University of Groningen, Nijenborgh 9, 9747 AG Groningen, The Netherlands}
\email{r.szabo@rug.nl, d.rodrigues.valesin@rug.nl}

\title[Removal of an edge in the contact process]{From survival to extinction of the contact process by the removal of a single edge}

\begin{abstract}
We give a construction of a tree in which the contact process with any positive infection rate survives but, if a certain privileged edge $e^*$ is removed, one obtains two subtrees in which the contact process with infection rate smaller than $1/4$ dies out.
\end{abstract}

\subjclass[2010]{}


\maketitle

\section{Introduction}
In this paper, we present an example of interest to the discussion of how the behaviour of interacting particle systems can be affected by local changes in the graph on which they are defined.

The \textit{contact process} on a locally finite and connected graph $G=(V,E)$ with rate $\lambda\geq 0$ is a continuous-time Markov process $(\xi_t)_{t\geq 0}$ with state space $\{0,1\}^V$ and generator
\begin{equation}\label{eq:generator}\mathcal{L}f(\xi) = \sum_{\substack{x \in V:}\\\xi(x) = 1} \left[f(\xi^{0\to x}) - f(\xi)  + \lambda \cdot \sum_{y \in V: \{ x,y \}\in E}  \left(f(\xi^{1\to y}) - f(\xi) \right)\right],\end{equation}
where $f$ is a local function, $\xi \in \{0,1\}^V$ and, for $i \in \{0,1\}$ and $z \in V$,
$$\xi^{i \to z}(w) = \begin{cases} i,&\text{if } w= z;\\\xi(w),&\text{otherwise},\end{cases}\qquad w \in V.$$
This process is usually seen as a model of epidemics: vertices are individuals, which can be healthy (state 0) or infected (state 1); infected individuals recover with rate 1 and transmit the infection to each neighbor with rate $\lambda$. A comprehensive exposition of the contact process can be found in \cite{L99}.

Given $A \subset V$, we denote by $(\xi^A_t)_{t \ge 0}$ the contact process with  initial configuration $\xi^A_0 = \mathds{1}_A$, the indicator function of $A$; if $A = \{x\}$, we write $\xi^x_t$ instead of $\xi^{\{x\}}_t$. We abuse notation and associate a configuration $\xi \in \{0,1\}^V$ with the set $\{x:\xi(x) = 1\}$. 

The contact process admits a well-known \textit{graphical construction}, which we now briefly describe. We let $\P^\lambda_G$ be a probability measure under which a family of independent Poisson point processes on $[0,\infty)$ are defined:
\begin{align*}
&D_x\text{ for } x \in V,\; \text{each with rate 1},\\
&D_{(x,y)}\text{ for } x,y \in V \text{ with }\{x,y\} \in E, \; \text{each with rate }\lambda;
\end{align*}
we regard each $D_x$ and $D_{(x,y)}$ as a random discrete subset of $[0,\infty)$. Given a realization of all these processes, an \textit{infection path} is a function $\gamma: [t_1,t_2] \to V$ which is right continuous with left limits and satisfies, for all $t \in [t_1, t_2]$,
$$t \notin D_{\gamma(t)}  \qquad \text{ and }\qquad  \gamma(t) \neq \gamma(t-) \text{ implies } t \in D_{(\gamma(t-),\gamma(t))}.$$
We say that two points $(x,s), (y,t) \in V\times [0,\infty)$ with $s \leq t$ are connected by an infection path if there exists an infection path $\gamma: [s,t] \to V$ with $\gamma(s) = x$ and $\gamma(t) = y$. This event is denoted by $\{(x,s) \leftrightarrow (y,t)\}$. Then, given $A \subseteq V$, the process
$$\xi^A_t(x) = \mathds{1}\{\exists y \in A: (y,0) \leftrightarrow (x,t) \}$$
has the distribution of the contact process with initial configuration $\mathds{1}_A$. 

To motivate our result, we will now state some facts which follow immediately either from the generator expression \eqref{eq:generator} or the graphical construction.  First, the ``all healthy'' configuration, represented as the empty set $\varnothing$, is a trap state for the dynamics. Second,
\begin{equation}\label{eq:x_to_y}
\P_G^\lambda\left(\exists t: \xi^A_t = \varnothing\right) \geq \P_{G'}^{\lambda'}\left(\exists t: \xi^B_t = \varnothing\right)\quad \text{ if } \lambda \leq \lambda',\;A \subseteq B \text{ and } G \subseteq G'
\end{equation}
($G \subseteq G'$ means that the vertex set and edge set of $G$ are respectively contained in the vertex set and edge set of $G'$). Third (using the fact that $G$ is connected),
\begin{equation}
\label{eq:trans_x_y}\P_G^\lambda\left(\exists t: B\subseteq \xi^A_t \right) > 0 \quad \text{ for all finite } A,B\subseteq V.
\end{equation}
Combining \eqref{eq:x_to_y} and \eqref{eq:trans_x_y}, it is seen that the probability
\begin{equation}\label{eq:prob_ext}\P_G^\lambda\left(\exists t \geq 0:\;\xi^A_t = \varnothing\right)\end{equation}
is either equal to 1 for any finite $A \subseteq V$ or strictly less than 1 for any finite $A \subseteq V$. The process is said to \textit{die out} (or \textit{go extinct}) in the first case and to \textit{survive} in the latter.

Whether one has survival or extinction may depend on both $G$ and $\lambda$, so one defines the critical rate
$$\lambda_c(G) = \sup\{\lambda: \P^\lambda_G\left(\exists t:\;\xi^A_t = \varnothing\right) = 1 \; \forall A \subseteq V, \;A \text{ finite}\}.$$
It follows from this definition and \eqref{eq:x_to_y} that the process dies out when $\lambda < \lambda_c$ and survives when $\lambda > \lambda_c$.

It is natural to expect that the critical rate $\lambda_c$ of the contact process is not affected by local changes on $G$, such as the addition or removal of edges (as long as $G$ remains connected). More precisely,
\begin{conjecture} [Pemantle and Stacey, \cite{PS01}]\label{conj1} 
Assume $G = (V,E)$ and $G' = (V,E')$ are two connected graphs with the same vertex set $V$ and $E' = E \cup \{\{x,y\}\}$, with $x,y \in V$. Then, $\lambda_c(G) = \lambda_c(G')$.
\end{conjecture} 
Jung \cite{J05} proved this conjecture for vertex-transitive graphs (a graph $G$ is vertex transitive if, for any two vertices $x$ and $y$, there exists an automorphism in $G$ that maps $x$ into $y$). Proving the conjecture in full generality is still an open problem.

Rather than making progress on this problem, we consider a slightly different line of inquiry. Let $G_1 = (V_1, E_1)$ and $G_2 = (V_2, E_2)$ be two graphs with disjoint  vertex sets $V_1$ and $V_2$. Let $x \in V_1$, $y \in V_2$ and define $G = (V,E)$ with $V = V_1 \cup V_2$ and $E = E_1 \cup E_2 \cup\{\{ x,y\} \}$ (that is, we connect the two graphs using the edge $\{ x,y\}$). It follows from \eqref{eq:x_to_y} that $\lambda_c(G) \leq \min(\lambda_c(G_1),\lambda_c(G_2))$, and it is natural to ask whether or not the inequality can be strict. Strict inequality would mean that, for some $\lambda < \min(\lambda_c(G_1),\lambda_c(G_2))$, the contact process with rate $\lambda$ on $G$ survives. This leads to a curious situation: since the process is subcritical on both $G_1$ and $G_2$, under $\P^\lambda_G$ there are almost surely no infinite infection paths entirely contained  either in  $G_1$ or in $G_2$, so any infinite infection path in $G$ needs to traverse the edge $\{ x,y\}$ infinitely many times.

We present an example of a graph in which this situation indeed occurs:


\begin{thm}\label{thm:main}
There exists a tree $G = (V,E)$ with a privileged edge $e^*$ so that
\begin{equation}\label{eq:thmeq1}\lambda_c(G) = 0\end{equation}
and, letting $G_1$, $G_2$ be the two subgraphs of $G$ obtained by removing $e^*$, we have 
\begin{equation} \lambda_c(G_1), \lambda_c(G_2) \geq \frac14.\label{eq:thmeq2}\end{equation}
\end{thm}

We end this Introduction discussing some related works in the interacting particle systems literature (apart from the already mentioned \cite{J05}). In \cite{MSS94}, Madras, Schinazi and Schonmann considered the contact process on $\mathbb{Z}$ in  deterministic inhomogeneous environments -- for them, this means that the recovery rates (that is, the rates of transition from state 1 to state 0) are vertex-dependent and deterministic, while the infection rate is the same everywhere. Among other results, they showed that if the recovery rate is equal to 1 everywhere except for a sufficiently sparse set $S \subset \mathbb{Z}$, where it is equal to some other value $b \in (0,1)$, then the critical infection rate $\lambda_c$ is the same as that of the original process on $\mathbb{Z}$. In \cite{NV96}, Newman and Volchan studied a contact process on $\mathbb{Z}$ in an environment in which the recovery rates are chosen randomly, independently among the vertices (the infection rate is again constant). They give a condition for the recovery rate distribution under which the process survives for any value of the infection rate (similarly to what happens to our graph $G$ of Theorem \ref{thm:main}). In  \cite{H03}, Handjani exhibited a modified version of  the voter model (which is another class of interacting particle system) in which modifications of the flip mechanism in a single site can change the probability of survival of the set of 1's from zero to positive.  

\section{Notation and preliminary results}
\label{sec:not}
Given a set $A$, the indicator function of $A$ is denoted $\mathds{1}_A$ and the cardinality of $A$ is denoted $|A|$.

Given a graph $G = (V,E)$, the degree of  $x\in V$ is denoted $\deg_G(x)$, the graph distance between  $x, y\in V$ is $\text{dist}_G(x,y)$ and the ball of radius $R$ with center $x$ is $B_G(x, R)$. We omit $G$ from the notation when it is clear from the context. We sometimes abuse notation and associate $G$ with its set of vertices (so that, for example, $|G|$ denotes the number of vertices of $G$).  A \textit{star graph} $S$ with hub $o$ on $n$ vertices is a tree with one internal node ($o$) and $n-1$ leaves. 

We always assume that the contact process is constructed from the graphical construction. Given $A, B\subseteq V$, $J_1, J_2\subseteq[0, \infty)$, we write $A\times J_1\leftrightarrow B\times J_2$ if $(x, t_1)\leftrightarrow(y, t_2)$ holds for some $x\in A$, $y\in B$, $t_1\in J_1$ and $t_2 \in J_2$.

In the remaining part of this section we will describe five preliminary results that will be needed in the proof of Theorem \ref{thm:main}. We start with the following.
\begin{lemma}\label{lem:rw_interval}
For any $\lambda \leq \frac14$, letting $I_n = \{1,\ldots, n\}$, we have
\begin{equation}\P^\lambda_\mathbb{Z}\left(\xi^{I_n}_t \subseteq I_n \; \forall t \right) \geq \frac12.\label{eq:rw_interval}\end{equation}
\end{lemma}
\begin{proof}
Define
$$L_t = \inf\{x: \xi^{I_n}_t(x) = 1\},\quad R_t = \sup\{x: \xi^{I_n}_t(x) = 1\},\quad t \geq 0,$$
with $\inf \varnothing = \infty$ and $\sup \varnothing = -\infty$. It is readily seen that $R_t$ is stochastically smaller than the continuous-time Markov chain $(X_t)$ on $\mathbb{Z}$ with $X_0 = n$ which jumps one unit to the right with rate $\lambda$ and jumps one unit to the left with rate 1. Hence,
$$\P^\lambda_\mathbb{Z}\left(R_t < n+1 \;\forall t \right) \geq \mathbb{P}\left(X_t < n+1 \;\forall t \right) \geq \frac34,$$
by an elementary computation for biased random walk on $\mathbb{Z}$.
We similarly have $\P^\lambda_\mathbb{Z}\left(L_t > 0 \;\forall t \right) \geq  \frac34$, so
$$\P^\lambda_\mathbb{Z}\left(\xi^{I_n}_t \subseteq I_n \; \forall t \right)  =  \P^\lambda_\mathbb{Z}\left(R_t < n+1 \text{ and } L_t > 0 \;\forall t \right) \geq \frac12.$$
\end{proof}

Our remaining four preliminary results are taken from \cite{MVY13}. The following shows that the contact process survives on a large star graph $S$ for a time that is exponential in $\lambda^2 |S|$. It is a refinement of the first result to this effect that appeared in \cite{BBCS05} in Lemma 5.3.   
\begin{lemma}(\cite{MVY13}, Lemma 3.1)\label{lemma:starsurv}
 There exists $\overline{c}>0$ such that, if $\lambda<1$, $S$ is a star with hub $o$ so that $\deg(o)>64e^2\cdot \frac{1}{\lambda^2}$ and $|\xi_0|>\frac 1 {16e}\cdot\lambda\deg(o)$, then 
\begin{equation}\label{eq:starsurv2}
 \P_S^{\lambda}(\xi_{e^{\overline{c}\lambda^2\deg(o)}}\neq\varnothing)\geq1-e^{-\overline{c}\lambda^2\deg(o)}.
\end{equation}
\end{lemma}
\noindent In \cite{MVY13}, this lemma was applied to guarantee that in a connected graph $G$ an infection around a vertex with sufficiently high degree is maintained long enough to produce an infection path that reaches another vertex at a certain distance:

\begin{lemma}(\cite{MVY13}, Lemma 3.2)\label{lemma:starreach}
 There exists $\lambda_0>0$ such that, if $0<\lambda<\lambda_0$, the following holds. If $G$ is a connected graph and $x, y$ are distinct vertices of $G$ with
  \begin{displaymath}
 \deg(x)>\frac 7 {\overline{c}} \frac 1 {\lambda^2} \log \left( \frac 1 {\lambda} \right) \cdot \textnormal{dist}_G(x, y) \text{ and } \frac{|\xi_0 \cap B(x, 1)|}{\lambda |B(x, 1)|}>\frac 1 {16e},
 \end{displaymath}
 then
  \begin{equation}\label{eq:starreach}
 \P_G^{\lambda}\left( \exists t: \frac{|\xi_t \cap B(y, 1)|}{\lambda |B(y, 1)|}>\frac 1 {16e}\right) > 1-2e^{-\overline{c} \lambda^2 \deg(x)}.
 \end{equation}
\end{lemma}
\noindent (Note that $\overline{c}$ is the same constant that appeared Lemma \ref{lemma:starsurv}).

In the opposite direction as Lemma \ref{lemma:starsurv}, the following result bounds from below the probability that the infection disappears from a star graph within time $3\log(1/\lambda)$:
\begin{lemma}(\cite{MVY13}, Lemma 5.2) \label{lemma:starext}
 If $\lambda<\frac 1 4$ and $S$ is a star, then
 \begin{equation}\label{eq:starext}
  \P_S^{\lambda} \left( \xi^S_{3\log\left( \frac 1 {\lambda} \right)} =\varnothing \right) \geq \frac 1 4 e^{-16\lambda^2 |S|}.
 \end{equation}
\end{lemma}
It is interesting to note that it follows from this result that, if $|S|$ is large, the infection will with high probability disappear before time $\exp\left\{C\lambda^2 |S| \right\}$ for any $C > 16$. This estimate on the time until the infection disappears thus matches (except for the value of the constant in the exponential) the one that follows from Lemma \ref{lemma:starsurv}.

It follows from Lemma \ref{lemma:starsurv} that vertices of degree much larger than $\frac{1}{\lambda^2}$ will sustain the infection for a long time. The last preliminary lemma in our list deals with tree graphs in which such big vertices are absent; in this case, it is unlikely that the infection spreads:
\begin{lemma}(\cite{MVY13}, Lemma 5.1) \label{lemma:treeext}
 Let $\lambda<\frac 1 2$ and $T$ be a finite tree with maximum degree bounded by $\frac 1 {8\lambda^2}$. Then, for any $x, y \in T$ and $t > 0$,
 \begin{equation}\label{eq:treeext1}
  \P_T^{\lambda}(\xi_t^T\neq\varnothing)\leq |T|^2 \cdot e^{-t/4} \text{ and }
 \end{equation}
 \begin{equation}\label{eq:treeext2}
 \P_T^{\lambda}( \{x\}\times[0,t]\leftrightarrow\{y\}\times\mathbb{R}_+ )\leq (t+1) \cdot (2\lambda)^{\mathrm{dist}_T(x, y)}.
 \end{equation}
\end{lemma}

\section{Proof of Theorem \ref{thm:main}}

\subsection{Construction of $G$}
Our graph $G$ will be equal to the one-dimensional lattice $\mathbb{Z}$ with the modification that a few vertices, denoted $o_1, o_2,\ldots$, are given extra neighbors, so that their degrees become increasingly large. The extra neighbors are vertices which we add to the graph as leaves.

We start defining sequences of integers $(o_i)_{i \geq 1}$, $(d_i)_{i \geq 1}$ satisfying
$$0 > o_1 > o_3 > \cdots,\qquad 0 < o_2 < o_4 < \cdots,\qquad 0 < d_1 < d_2 < \cdots. $$
The definition will be inductive. We set $d_1 = 1$, $o_1 = -1$ and $o_2 = 2$. Assume we have already defined $o_1,\ldots, o_i$ and $d_1, \ldots, d_{i-1}$. Then, set
\begin{align}\label{eq:choice_od}
o_{i+1} = \begin{cases}
o_{i-1} + i\cdot \sum_{j=1}^{(i-1)/2}d_{2j} &\text{if $i$ is odd},\\[.2cm]
o_{i-1} - i\cdot  \sum_{j=0}^{(i-2)/2} d_{2j+1}&\text{if $i$ is even,}
\end{cases}\qquad\qquad d_i = i\cdot  |o_i - o_{i+1}|.
\end{align}
We clearly have
\begin{equation}
d_i > i! \label{eq:di_large}
\end{equation}

Now, for each $i \geq 1$, let $\{x^i_{1},\ldots,x^i_{d_i}\}$ be a set with $d_i$ distinct elements (for distinct values of $i$, these sets are assumed to be disjoint). Then let
$$G = (V, E), \text{ with }\quad  V = \mathbb{Z} \cup \bigcup_{i=1}^\infty \{x^i_1,\ldots, x^i_{d_i}\},\quad E = E(\mathbb{Z}) \cup \bigcup_{i=1}^\infty \bigcup_{j=1}^{d_i} \{\{ o_i, x^i_j\}\},$$
where $E(\mathbb{Z})$ is the set of edges of $\mathbb{Z}$. The construction is illustrated on Figure~\ref{fig:graph}.
We let $e^*$ be the edge $\{ 0, 1\}$. When $e^*$ is removed, $G$ is split into two subgraphs: we let $G_-$ denote the one associated to the negative half-line, and $G_+$ the one associated to the positive half-line.

\begin{figure}[h]
\begin{center}
\def\svgwidth{450pt}
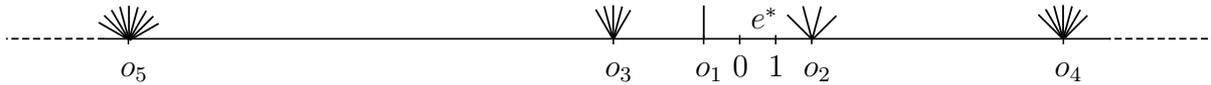\qquad
\caption{The graph $G$. \label{fig:graph}}
\end{center}
\end{figure}

The strategy in the construction of $G$, and in particular in the choice of $(d_i)$ and $(o_i)$, is as follows.
\begin{itemize}
\item As we will prove in the next subsection, for any $\lambda > 0$, if $i$ is sufficiently large, the star $B(o_i,1)$ is large enough to sustain the infection long enough that it reaches $B(o_{i+1},1)$ with high probability. The infection is then sustained there long enough to reach $B(o_{i+2},1)$ and so on. As the probability of the intersection of all these events is close to 1, there is survival. Note that, as observed in the Introduction, indeed the infection necessarily relies on infinitely many traversals of $e^*$ in order to survive.
\item In Subsection \ref{sub:ext}, we will show that, if $\lambda = \frac14$ and $e^*$ is absent, then the star $B(o_i, 1)$ is not quite large enough to hold the infection long enough to overcome the distance to $o_{i+2}$. Hence, it becomes increasingly difficult for the infection to travel from one star to the next in the same half-line, and consequently there is extinction.
\end{itemize}

\subsection{Proof of $\lambda_c(G) = 0$.} We need to show that, for any $\lambda > 0$, the contact process with rate $\lambda$ on $G$ survives. By \eqref{eq:x_to_y}, it is sufficient to show this for $\lambda \in (0,\lambda_0)$, where $\lambda_0$ is as in Lemma \ref{lemma:starreach}. 

Fix $\lambda \in (0,\lambda_0)$. As explained in the Introduction, it is enough to show that there exists a finite set $A \subset V$ such that \eqref{eq:prob_ext} is strictly less than 1. Assume $i \in \mathbb{N}$ is large enough that
\begin{equation*}
 \deg(o_i) = d_i + 2 > i \cdot |o_i - o_{i+1}| = i\cdot \text{dist}_G(o_i,o_{i+1}) >\frac 7 {c} \frac 1 {\lambda^2} \log \left( \frac 1 {\lambda} \right) \cdot \text{dist}_G(o_i,o_{i+1}).
 \end{equation*}
Then, by Lemma \ref{lemma:starreach},
\[\text{if } A \subseteq V,\;\frac{|A \cap B(o_i,1)|}{\lambda |B(o_i,1)|} > \frac{1}{16e},\quad \text{ then }
 \P_{G}^{\lambda}\left( \exists t: \frac{|\xi_t^{A} \cap B(o_{i+1},1)|}{\lambda \cdot |B(o_{i+1},1)|}>\frac 1 {16e} \right)>1-2e^{-c \lambda^2 d_i}.
\]
The desired result now follows from \eqref{eq:di_large}, the Strong Markov Property and a union bound.

\subsection{Proof of $\lambda_c(G_-),\lambda_c(G_+) \geq \frac14$.}\label{sub:ext}
We will only carry out the proof of $\lambda_c(G_+) \geq \frac14$; the proof for $G_-$ is similar. As explained in the Introduction, it is sufficient to show that, in the contact process on $G_+$ with rate 
\begin{equation}\label{eq:choice_of_lambda}\lambda = \frac14\end{equation} and started with a single infection located at vertex 1, the infection almost surely disappears:
\begin{equation}
\P^\lambda_{G_+}\left(\xi^1_t \neq \varnothing \;\forall t\right) = 0.\label{eq:want_to_prove}
\end{equation}

Define the sets of vertices
\begin{align*}&S_j = \{o_j, x^j_1,\ldots, x^j_{d_j}\}, \quad j \in \{2,4,\ldots\}\\
&H_0 = \{1\},\\
&H_j = (o_j, o_{j+2}) \cap \mathbb{Z}, \quad j \in \{2,4,\ldots\},\\
&G_i = H_0 \cup \left(\bigcup_{j=1}^{i/2} (S_{2j} \cup H_{2j})\right), \quad i \in \{2,4,\ldots\}.
\end{align*} 
We will abuse notation and refer to the above sets as subgraphs of $G_+$; for instance, $H_2$ will be the subgraph with vertex set defined above and set of edges having both extremities in this vertex set.

We now fix an arbitrary $i \in 2\mathbb{N}$. Define
\begin{equation}\label{eq:def_of_tau}
\tau = \exp \left\{\frac{i}{2}(d_2 +d_4 +\cdots + d_i) \right\}.
\end{equation}
We have
\begin{align}
\nonumber &\P^\lambda_{G_+}\left(\xi^1_t \neq \varnothing \;\forall t\right) \leq \P^\lambda_{G_+}\left(\xi^1_{\tau} \neq \varnothing\right)
\\&\nonumber \leq \P^\lambda_{G_+}\left(\xi^1_{\tau} \neq \varnothing,\; \xi^1_t \subseteq G_i \;\forall t \leq \tau\right) + \P^\lambda_{G_+}\left(\xi^1_{\tau} \neq \varnothing,\; \xi^1_t \nsubseteq G_i \;\text{ for some } t \leq \tau\right)
\\&\leq \P^\lambda_{G_i}\left(\xi^{G_i}_{\tau} \neq \varnothing\right) + \P^\lambda_{G_+}\left(\{o_i\}\times [0,\tau] \leftrightarrow \{o_{i+2}\} \times [0,\tau] \right).\label{eq:two_terms}
\end{align}

We bound the two terms on \eqref{eq:two_terms} separately, starting with the second:
\begin{align}
 \nonumber&\P^\lambda_{G_+}\left(\{o_i\}\times [0,\tau] \leftrightarrow \{o_{i+2}\} \times [0,\tau] \right)\\&\hspace{2cm}  \nonumber= \P^\lambda_{H_i \cup \{o_i,o_{i+2}\}}\left(\{o_i\}\times [0,\tau] \leftrightarrow \{o_{i+2}\} \times [0,\tau] \right) \\&\hspace{2cm} \nonumber\stackrel{\eqref{eq:treeext2}}{\leq} (\tau + 1)\cdot (2\lambda)^{\text{dist}(o_i,o_{i+2})}\\
 &\hspace{2cm} \stackrel{\eqref{eq:choice_od},\eqref{eq:choice_of_lambda},\eqref{eq:def_of_tau}}{\leq} 2 \exp\left\{i(d_2 +d_4 +\cdots + d_i) \left(\frac12 - \log(2) \right) \right\} < \exp\left\{-d_i\right\}\label{eq:very_last1}
\end{align}
if $i$ is large enough.

We now turn to the first term in \eqref{eq:two_terms}. First define
\begin{equation}\label{eq:def_times}
t_1 =3 \log \left(\frac 1 {\lambda}\right),\qquad  L = i\log(\text{dist}(o_i,o_{i+2})) = i\log\left(i(d_2 + d_4 + \cdots + d_i) \right).
\end{equation}
We will assume that $i$ is large enough (depending on $\lambda$) so that $L > t_1$.
Using the Markov property and \eqref{eq:x_to_y}, we have
\begin{equation}
\P^\lambda_{G_i}\left(\xi^{G_i}_{\tau} \neq \varnothing \right) \leq \P^\lambda_{G_i}\left(\xi^{G_i}_{L} \neq \varnothing \right)^{\lfloor \tau/L \rfloor}.
\label{eq:subadd}\end{equation}
We then bound
\begin{equation}\begin{split}
 &\P^\lambda_{G_i}\left(\xi^{G_i}_{L} = \varnothing \right) \\&\qquad\geq \prod_{j=1}^{i/2} \P^\lambda_{G_i}\left(\xi^{S_{2j}}_t \subseteq S_{2j} \;\forall t,\; \xi^{S_{2j}}_{t_1} = \varnothing \right) \cdot \prod_{j=0}^{i/2} \P^\lambda_{G_i}\left(\xi^{H_{2j}}_t \subseteq H_{2j} \;\forall t,\; \xi^{H_{2j}}_{L} = \varnothing \right).\end{split}\label{eq:two_more_terms}
\end{equation}
Now, for all $j \leq i/2$,
\begin{align}\nonumber
\P^\lambda_{G_i}\left(\xi^{H_{2j}}_t \subseteq H_{2j} \;\forall t,\; \xi^{H_{2j}}_{L} = \varnothing \right) & \geq \P^\lambda_{G_i}\left(\xi^{H_{2j}}_t \subseteq H_{2j} \;\forall t \right) - \P^\lambda_{H_{2j}}\left(\xi^{H_{2j}}_L \neq \varnothing \right)\\
\nonumber&\stackrel{\eqref{eq:rw_interval},\eqref{eq:treeext1}}{\geq} \frac12 - |H_{2j}|^2\cdot e^{-L/4} \\& \stackrel{\eqref{eq:def_times}}{=} \frac12 - (\text{dist}(o_i,o_{i+2}))^{2-i/4} \geq \frac14\label{eq:final_est1}
\end{align} 
if $i$ is large enough. Again for all $j \leq i/2$, we have
\begin{align}
&\nonumber\P^\lambda_{G_i}\left(\xi^{S_{2j}}_t \subseteq S_{2j} \;\forall t,\; \xi^{S_{2j}}_{t_1} = \varnothing \right) \\&\nonumber\hspace{2cm}\geq \P^\lambda_{S_{2j}}\left(\xi^{S_{2j}}_{t_1} = \varnothing \right) \cdot \P^\lambda_{G_i}\left(\left(D_{\{ o_{2j},o_{2j}-1\}}\cup D_{\{ o_{2j},o_{2j}+1\}}\right) \cap [0,t_1] = \varnothing \right)\\
&\hspace{2cm} \nonumber\stackrel{\eqref{eq:starext}}{\geq} \frac14\exp\left\{-16\lambda^2|S_{2j}| \right\}\cdot \exp\left\{-2\lambda t_1 \right\}\\
&\hspace{2cm} \stackrel{\eqref{eq:def_times}}{\geq} \frac{\lambda^{6\lambda}}{4} \exp\left\{-17\lambda^2 d_{2j}\right\}.\label{eq:final_est2}
\end{align}
Using \eqref{eq:final_est1} and \eqref{eq:final_est2} in \eqref{eq:two_more_terms}, we get
\begin{align*}\P^\lambda_{G_i}\left(\xi^{G_i}_{L} = \varnothing \right) &\geq \left(\frac{\lambda^{6\lambda}}{16}\right)^{\frac{i}{2}+1} \cdot \exp\left\{-17\lambda^2 (d_2 + d_4 + \cdots + d_{i}) \right\} \\&\stackrel{\eqref{eq:di_large}}{\geq}  \exp\left\{-18\lambda^2 (d_2 + d_4 + \cdots + d_{i}) \right\}\end{align*}
if $i$ is large enough; using this in \eqref{eq:subadd}, we get
\begin{align}\nonumber\P^\lambda_{G_i}\left(\xi^{G_i}_{\tau} \neq \varnothing \right) &\leq \exp\left\{-\exp\left\{-18\lambda^2(d_2 + d_4+\cdots + d_i) \right\}\cdot \frac{\exp\left\{\frac{i}{2}(d_2 + d_4+\cdots + d_i) \right\}}{i\log\left(i(d_2 + d_4 + \cdots + d_i) \right)} \right\} \\
&<\exp\{-d_i\} \label{eq:very_last2}
\end{align}
if $i$ is large enough.

In conclusion, using \eqref{eq:very_last1} and \eqref{eq:very_last2} in \eqref{eq:two_terms}, we see that $\P^\lambda_{G_+}\left(\xi^1_t \neq \varnothing \; \forall t\right) < 2\exp\{-d_i\}$ for all $i$, so \eqref{eq:want_to_prove} follows.

\vspace{.5cm}
\begin{center}\textbf{Acknowledgments}
\end{center}
The authors would like to thank Ori Gurel-Gurevich, F\'abio Machado and Pablo Rodriguez for helpful discussions and Aernout van Enter for suggestions on the writing of this work.


\end{document}

%% file: graph.eps_tex
\begingroup%
  \makeatletter%
  \providecommand\color[2][]{%
    \errmessage{(Inkscape) Color is used for the text in Inkscape, but the package 'color.sty' is not loaded}%
    \renewcommand\color[2][]{}%
  }%
  \providecommand\transparent[1]{%
    \errmessage{(Inkscape) Transparency is used (non-zero) for the text in Inkscape, but the package 'transparent.sty' is not loaded}%
    \renewcommand\transparent[1]{}%
  }%
  \providecommand\rotatebox[2]{#2}%
  \ifx\svgwidth\undefined%
    \setlength{\unitlength}{257.69748586bp}%
    \ifx\svgscale\undefined%
      \relax%
    \else%
      \setlength{\unitlength}{\unitlength * \real{\svgscale}}%
    \fi%
  \else%
    \setlength{\unitlength}{\svgwidth}%
  \fi%
  \global\let\svgwidth\undefined%
  \global\let\svgscale\undefined%
  \makeatother%
  \begin{picture}(1,0.10023145)%
    \put(0,0.01){\includegraphics[width=\unitlength]{graph.eps}}%
    \put(0.61933,0.025){\color[rgb]{0,0,0}\makebox(0,0)[lb]{\smash{$e^*$}}}%
    \put(0.095,-0.015){\color[rgb]{0,0,0}\makebox(0,0)[lb]{\smash{$o_5$}}}%
    \put(0.498,-0.015){\color[rgb]{0,0,0}\makebox(0,0)[lb]{\smash{$o_3$}}}%
    \put(0.573,-0.015){\color[rgb]{0,0,0}\makebox(0,0)[lb]{\smash{$o_1$}}}%
    \put(0.60433,-0.015){\color[rgb]{0,0,0}\makebox(0,0)[lb]{\smash{$0$}}}%
    \put(0.63366,-0.015){\color[rgb]{0,0,0}\makebox(0,0)[lb]{\smash{$1$}}}%
    \put(0.663,-0.015){\color[rgb]{0,0,0}\makebox(0,0)[lb]{\smash{$o_2$}}}%
    \put(0.872,-0.015){\color[rgb]{0,0,0}\makebox(0,0)[lb]{\smash{$o_4$}}}%
  \end{picture}%
\endgroup%